\documentclass[a4paper,11pt]{article}
\usepackage[english]{babel}
\usepackage{graphicx}

\usepackage{amsxtra}
\usepackage{amssymb}
\usepackage{amsmath}
\usepackage{amsthm}
\usepackage{color}
\usepackage{hyperref}

\newtheorem{prop}{Proposition}[section]
\newtheorem{define}[prop]{Definition}

\newtheorem{theo}[prop]{Theorem}

\numberwithin{equation}{section}

\theoremstyle{remark}
\newtheorem{rmq}{Remark}

\newcommand{\cxi}{\langle \xi\rangle}

\newcommand{\la}{\langle}
\newcommand{\ra}{\rangle}
\newcommand{\di}{\displaystyle}

\newcommand{\n}{\nabla}

\newcommand{\C}{\mathbb{C}}
\newcommand{\R}{\mathbb{R}}
\newcommand{\N}{\mathbb{N}}
\newcommand{\Z}{\mathbb{Z}}

\newcommand{\tX}{\widetilde{X}}

\newcommand{\tJ}{\widetilde{J}}
\newcommand{\tK}{\widetilde{K}}
\newcommand{\cxip}{\langle \xi'\rangle}
\newcommand{\cF}{\mathcal{F}}

\title{Interpolation of subspaces of infinite codimension}
\author{Corentin Audiard
\footnote{Sorbonne Universit\'e, Universit\'e Paris Cit\'e, CNRS, Laboratoire Jacques-Louis Lions, LJLL, F-75005 Paris, France}}
\begin{document}
\maketitle
\begin{abstract}Given three normed spaces $X_0,X_1,\widetilde{X_1}$ with 
$\widetilde{X_1}$ a closed subspace of $X_1$ with same norm, how does 
$[X_0,\widetilde{X_1}]_\theta$ compare to $[X_0,X_1]_\theta$ ?
This classical problem of interpolation theory has been mostly considered 
in the case of finite codimensional subspaces. We 
show here that quotient $J$ norms are a useful tool which allows to tackle 
several relevant cases.
\end{abstract}
\renewcommand{\abstractname}{R\'esum\'e}
\begin{abstract}
\end{abstract}

\section{Introduction}
The interpolation of subspaces is a classical topic which originates at 
least to the introduction of the famous Lions-Magenes space 
$H^{1/2}_{00}(\R^{+*})$, defined as the interpolation space 
$[L^2(\R^{+*}),H^1_0(\R^{+*})]_{1/2,2}$ (a short reminder about real interpolation theory 
is included in section \ref{sec:definterp}). Here $L^2$ and $H^1$ are the usual Lebesgue 
and Sobolev spaces endowed with norms 
$$\|u\|_{L^2}=\left(\int_{0}^\infty|u|^2dx\right)^{1/2},\ \|u\|_{H^1}=\|u\|_{L^2}+\|u'\|_{L^2}.$$ 
$H^1_0(\R^{+*})$ is the closure of $\mathcal{C}_c^\infty(\R^{+*})$ for 
the $H^1$ norm. \\
While for $\theta\neq 1/2$, the interpolation space $[L^2,H^1_0]_\theta$ coincides with $H^{\theta}_0$ (the closure of $\mathcal{C}_c^\infty(\R^{+*})$ for the 
$[L^2,H^1]_\theta$ norm), 
the space $H^{1/2}_{00}$ is algebraically and topologically different 
from $H^{1/2}_0$, as its norm is given by 
$$
\|u\|_{H^{1/2}_{00}(\R^{+*})}=\|u\|_{H^{1/2}(\R^{+*})}+\left(\int_0^\infty\frac{|u(x)|^2}{x}dx\right)^{1/2}.
$$
This remarkable result is a special case of the general problem of interpolation of 
subspaces : if $X_0,X_1$ is an interpolation couple, and $\tX_1$ is a closed 
subspace of $X_1$, it is often interesting to identify $[X_0,\tX_1]_{\theta,p},\ 
1\leq p\leq \infty$, and 
in particular to know if it is a closed subspace of $[X_0,X_1]_{\theta,p}$.\\
In many cases, there is a natural guess : if $\tX_1$ is the kernel of a linear map 
$L$ defined on $X_1$ and which ``makes sense''
on $[X_0,X_1]_{\theta,p}$ for a range $\theta >\theta_0$, then 
one expects the following alternative:
\begin{enumerate}
\item For $\theta>\theta_0$, $[X_0,\tX_1]_{\theta,p}$ is the kernel of $L|_{[X_0,X_1]_{\theta,p}}$ (hence a closed subspace),
\item For $\theta<\theta_0$, $[X_0,\tX_1]_{\theta,p}=[X_0,X_1]_{\theta,p}$ ,
\item For $\theta=\theta_0$, $[X_0,\tX_1]_{\theta,p}$ has a different, stronger 
topology than $[X_0,X_1]_{\theta,p}$ (and is not the same set).
\end{enumerate}
This is precisely what happens in the example mentioned above, with $\theta_0=1/2$.
Interpolation of subspaces appears very naturally in the analysis of partial differential 
equations. The books of Lions and Magenes \cite{LionsMagenes,LionsMagenes2,LionsMagenes3}
and Triebel \cite{Triebel2} are major classical reference,
and since then there have been countless developments and applications, 
see for example the reference book of Amann \cite{Amann2} on parabolic PDEs, 
or the recent article of Dalibard et al \cite{Dalibach}.
\\
There is an abundant literature solving the 
problem of interpolation of Sobolev spaces with boundary conditions, 
e.g. Grisvard \cite{Grisvard},\cite{Grisvard2}, for a method 
more in the spirit of this article we refer to L\"ofstr\"om \cite{LofNeuman}. 
\\
The general abstract problem was considered by several authors, 
including L\"ofstr\"om \cite{LofsSub}, 
Ivanov-Kalton \cite{IvanovKalton}, Asekritova et al 
\cite{AsCoKru}.
The simplest case is the one where $X_1\subset X_0$, $\tX_1$ is the kernel of a 
continuous linear form $l$ (hence a subspace of codimension one) and $l$ satisfies the 
following kind of homogeneity property 
\begin{equation}\label{hom1}
\exists\,\theta_0\in ]0,1[:\ \inf \{J(t,u):=\|u\|_{X_0}+t\|u\|_{X_1},\ u\in X_1,\ l(u)=1\}\sim t^{\theta_0},
\end{equation}
where $A(t)\sim B(t)$ means that there are two positive constants $C_1,C_2$ 
independent of $t\in ]0,1]$ such that 
$C_1A(t)\leq B(t)\leq C_2A(t)$. The condition \eqref{hom1} is readily seen to be equivalent to 
\begin{equation}\label{hom2}
\sup_{J(t,u)=1}|l(u)|\sim_0 t^{-\theta_0},
\end{equation}
or to the following form, which is more familiar to users of functional spaces
\begin{equation}\label{hom3}
\left\{
\begin{array}{ll}
|l(t,u)|\lesssim \|u\|_{X_0}^{1-\theta_0}\|u\|_{X_1}^{\theta_0},
\\
\exists\,u:\ ]0,1]\to X_1,\ \forall\,0<t\leq 1, \|u(t)\|_{X_0}=1,\ \|u(t)\|_{X_1}=t\text{ and } |l(u(t))|\sim_0 \|u(t)\|_{X_1}^{\theta_0}.
\end{array}\right.
\end{equation}
Under these conditions, L\"ofstr\"om obtained the following general result, confirming 
the ``natural guess'' :
\begin{theo}[L\"ofstr\"om\cite{LofsSub}]\label{th:lof}
Let $X_1\subset X_0$ Banach spaces,  $l$  a  linear form $X_1\to \C$ which satisfies \eqref{hom1}, $\tX_1=Ker(l)$. Then
\begin{enumerate}
 \item For $0<\theta<\theta_0$, $[X_0,\tX_1]_{\theta,p}=[X_0,X_1]_{\theta,p}$ (with 
 same topology).
 \item For $1>\theta>\theta_0$, $l$ can be continuously extended to $[X_0,X_1]_{\theta,p}$ and 
 $$[X_0,\tX_1]_{\theta,p}=[X_0,X_1]_{\theta,p}\cap Ker(l).$$
\end{enumerate}
If moreover the couple $(X_0,X_1)$ is linearisable by a family of operators $\Lambda(t)$, then 
$[X_0,\tX_1]_{\theta_0,p}$ is not a closed subspace of $[X_0,X_1]_{\theta_0,p}$ and 
is identified to
$$
[X_0,\tX_1]_{\theta_0,p}=\left\{u\in [X_0,X_1]_{\theta_0,p}:\ \int_0^1\frac{|l(\Lambda(t)u)|^p}
{t^{1+\theta_0 p}}dt<\infty\}\right\},
$$
with norm $\|u\|_{[X_0,\tX_1]_{\theta_0,p}}=\|u\|_{[X_0,X_1]_{\theta_0,p}}+
\|l(\Lambda(t)u\|_{L^p([0,1],dt/t)}$.
\end{theo}
The actual result in \cite{LofsSub} is more general as it handles the case where $\tX_1$ is the kernel 
of a finite number of linear forms of different homogeneity (with some independence condition, see section \ref{sec:indep} for more details),  
we  state this version for simplicity. The condition that $\tX_1$ should be of codimension one was 
later lifted by Asekritova et al \cite{AsCoKru}, but the proof heavily relied on the result of 
L\"ofstr\"om and as such can not be generalized to infinite codimension. We also point out that in 
a long (unfortunately unpublished) article, L\"ofstr\"om gave an interpolation result for kernels of
infinite codimension. A (simplified) statement is the following : 
\begin{theo}[L\"ofstr\"om \cite{Lofmieux}]\label{th:lof2}
Let $(X_0,X_1,Y)$, $X_1\subset X_0$ Banach spaces, $f:X_1\to Y$ a linear map, 
$\tX_1=\text{Ker}(f)$, and assume that for $0<t\leq 1$ there exists a linear map 
$V(t):\ X_1\to X_1$ such that 
\begin{itemize}
 \item For $u\in X_1$, $f(V(t)u)=f(u)$,
 \item For $0<t\leq 1$, $0<s\leq 1/t$, $J(t,V(t)u)\lesssim s^{-\theta_0}J(st,u)$,
 \item For $t\in ]0,1]$, $\text{Ker}(V(t))=\text{Ker}(f)$\footnote{This assumption is not present in the original article, but we believe that the proof of the theorem 
 requires it.}.
\end{itemize}
Then the conclusion of theorem \ref{th:lof} holds (except the ``linearizable part'').
\end{theo}

Our aim is here is to provide several generalizations of theorem \ref{th:lof2},
stated in a way which is more intrinsic to the subspace $\tX_1$, and more precise, 
including the linearizable case which is very relevant in applications.\\
The article is organized as follows: we start in section \ref{sec:definterp}
with a quick reminder on interpolation, including the notion of linearisable couple, 
in section \ref{sec:exactorder} we give a 
notion of ``order of a subspace'', this allows to prove an extension of theorem \ref{th:lof2} . Then in section \ref{sec:linorder}
we give a refinement of this result in the linearisable case. 
We show that this refinement is essential to tackle classical and less classical 
examples for the interpolation of closed subspaces of Sobolev spaces 
(which was our original motivation in considering this problem).
Finally, in section \ref{sec:indep}, we extend our result to the case where 
$\tX_1$ is the intersection of subspaces with different order, under some 
``independence'' condition.
\paragraph{Acknowledgement} This work was partially funded by the ANR project 
ANR-24-CE40-3260, Hyperbolic Equations, Approximations and Dynamics (HEAD).
\section{Notations and reminder on interpolation}\label{sec:definterp}
By $A\lesssim B$, resp. $A\sim B$, we mean that there exists constants $C_1$, 
resp. $C_2,C_3$ independent of 
all parameters such that 
$$
A\leq C_1 B,\ \text{resp. } C_2A\leq B\leq C_3 B.
$$
\\
For details and proofs we refer to 
the book of Bergh and L\"ofstr\"om \cite{berglof}.\\
Throughout this article, $X_0,X_1$ are Banach spaces, with $X_1\subset X_0$, and 
$\|u\|_{X_1}\lesssim \|u\|_{X_0}$. $\tX_1$ is a closed subspace of $X_1$.
\\
The $K$ function is defined for $u\in X_0$, $t>0$ by 
$$
K(t,u,X_0,X_1)=\inf_{u=u_0+u_1\in X_0+X_1}\|u_0\|_{X_0}+t\|u_1\|_{X_1}.
$$
The $J$ function is defined for $u\in X_1$, $t>0$ by 
$$
J(t,u,X_0,X_1)=\|u\|_{X_0}+t\|u\|_{X_1}.
$$
For simplicity of notation, we will simply write $J(t,u,X_0,X_1)=J(t,u)$, 
$K(t,u,X_0,X_1)=K(t,u)$. Similarly, $\tJ$ and $\tK$ are the $J$ and $K$ functions associated 
to the couple $(X_0,\tX_1)$.
The interpolation space $[X_0,X_1]_{\theta,p}$ is defined as follows:
\begin{define}\label{def:interp}
 For $1\leq p\leq \infty$, $0<\theta<1$, 
 $$
 [X_0,X_1]_{\theta,p}=\left\{u\in X_0:\ \int_0^\infty\left(\frac{K(t,u)^p}{t^{\theta }}\right)^p
 \frac{dt}{t}<\infty\right\},
 $$
 it is equivalently defined as $\di \left\{u=\int_0^\infty \frac{v(t)}{t}dt:\ 
 \int_0^\infty \left(\frac{J(t,v(t)}
 {t^\theta}\right)^pdt\right\}$, with equivalent norms 
 $$
\left(\int_0^\infty\frac{K(t,u)^p}{t^{1+\theta p}}dt\right)^{1/p}\sim \inf_{u=\int_0^\infty vdt/t}
\left(\int_0^\infty\left(\frac{J(t,v(t)}
 {t^\theta}\right)^pdt\right)^{1/p}.
 $$
 Discrete equivalent norms are 
 $$
 \|u\|_{K,\theta,p}=\left\|(2^{\theta j}K(2^{-j},u))\right\|_{l^p(\Z)}, 
 \|u\|_{J,\theta,p}=\inf_{u=\sum_\Z v_j}\left\|(2^{\theta j}J(2^{-j},v_j))\right\|_{l^p(\Z)}
 $$
 If $X_1\subset X_0$, the integration interval can be replaced by $]0,1]$, and 
 the summation set can be replaced by $\N$. This gives equivalent norms.
\end{define}
Now the definition of linearisable spaces is the following : 
\begin{define}
When $X_1\subset X_0$, the couple $(X_0,X_1)$ is linearisable when there exists 
maps $(\Lambda(t))_{0<t\leq 1}$, 
 $X_0\to X_1$  such that for any $0<t\leq 1$, $u\in X_0$
 $$
 K(t,u)\sim \|u-\Lambda(t)u\|_{X_0}+t\|\Lambda(t)u\|_{X_1}.
 $$
 This is equivalent to the bound 
 \begin{equation}\label{boundLambda}
 \|u-\Lambda(t)u\|_{X_0}+t\|\Lambda(t)u\|_{X_1}\lesssim \min(\|u\|_{X_0},t\|u\|_{X_1}).
 \end{equation}
\end{define}

\section{Order of a subspace, applications}\label{sec:exactorder}
We recall that for simplicity we assume $X_1\subset X_0$, and $J(t,u)$ 
is defined as $\|u\|_{X_0}+t\|u\|_{X_1}$. 
Given $\tX_1$ a closed 
subspace of $X_1$, we denote the quotient $J$ norm 
$$
N_{J,\tX_1}(t,u)=\inf_{v\in \tX_1}J(t,u-v,X_0,X_1).
$$
(when the functional framework is clear, we will usually write $N_J$ instead of 
$N_{J,\tX_1}$).\\
Quotient norms are convenient to give an intrinsic notion of order for subspaces:
\begin{define}
 We say that a continuous linear map $f:X_1\to Y$ is exactly of order $\theta_0$ if 
 $$
 \|f(u)\|_Y\sim_0 t^{-\theta_0}N_{J,Ker(f)}(t,u-v).
 $$
 A subspace $\tX_1$ of $X_1$ is exactly of order $\theta_0$ when the projection operator 
 $ \Pi :\ X_1\to X_1/\tX_1$ is exactly of order $\theta_0$, namely
\begin{equation}\label{eq:caracexact}
N_J(t,u):=\inf_{v\in \tX_1} J(t,u-v)\sim_0 
t^{\theta_0}\|u\|_{X_1/\tX_1}.
\end{equation}
 \end{define}
This is indeed an extension of the settings of theorem \ref{th:lof2}:
\begin{prop} 
Let $f$ a continuous linear map $X_1\to Y$ with kernel 
$\tX_1:=\text{Ker}(f)$. If there exists $(V(t))_{0<t\leq 1}$ as 
in theorem \ref{th:lof2}, then by quotient they induce maps $X_1/\tX_1\to X_1$ which 
are right inverses of the projection $\Pi:\ X_1\to \tX_1$, and 
 $\tX_1$ is exactly of order $\theta_0$.
 \\
 Conversely, if $\tX_1$ is a subspace of $X_1$ exactly of order $\theta_0$ and the projection $X_1\to X_1/\tX_1$ has 
 right inverses $R(t)$ with
$\|R(t)\|_{(X_1/\tX_1,N_J(t,\cdot))\to (X_1,J(t,\cdot)})$ bounded uniformly for 
$0<t\leq 1$, then the family 
$(R(t)\circ \Pi)$ is a family of operators as in theorem \ref{th:lof2}.
 \end{prop}
 \paragraph{Remark} In general the right inverses $R(t)$ have no reason to exist, 
 since their existence imply that $\tX_1$ is supplemented in $X_1$,
 hence our framework is slightly more general than in theorem \ref{th:lof2}.
\begin{proof}
Since $\tX_1=\text{Ker}(V(t))$, $V(t)$ induces by quotient a linear map $\widetilde{V}:\ 
X_1/\tX_1\to X_1$. 
Moreover for $u\in X_1$, since $f(V(t)u)=f(u)$, we have $V(t)u-u\in \text{Ker}(f)=\tX_1$,
so if $\Pi$ is the projection $X_1\to X_1/\tX_1$, then $\Pi\circ \widetilde{V}(t)=I_d:\ 
X_1/\tX_1\to X_1/\tX_1$, and we have by the second point of the 
definition of $V$
$$
N_J(t,u)=N_J(t,V(t)u)=N_J(t,V^2(t)u)\lesssim t^{\theta_0} N_J(1,V(t)u)\sim t^{\theta_0}\|u\|_{X_1/\tX_1}.
$$
The converse inequality is readily obtained with the same computation :
$$
\|u\|_{X_1/\tX_1}\sim N_J(1,V(1)u)=N_J(1,V(t)V(1)u)\lesssim t^{-\theta_0}N_J(t,u). $$
Next if there exists right inverses $(R(t))_{0<t\leq 1}$ to $\Pi$, then for any 
$u\in X_1$, $u-R(t)\Pi u\in \tX_1$, so we have 
by definition $N_J(t,u)\leq J(t,R(t)\Pi u)$ and by assumption
$J(t,R(t)\Pi u)\lesssim N_J(t,u)$ so 
$$
N_J(t,u)\sim J(t,R(t)\Pi u).
$$
This gives directly from \eqref{eq:caracexact}, for $0<s\leq 1/t$
\begin{eqnarray*}
J(t,R(t)\Pi u)\sim N_J(t,u)&\sim& t^{\theta_0}\|u\|_{X_1/\tX_1}
\\
&=&(st)^{\theta_0}\|u\|_{X_1/\tX_1}s^{-\theta_0}
\\
&\sim& N_J(st,u)s^{-\theta_0}
\\
&\lesssim &
s^{-\theta_0}J(st,u).
\end{eqnarray*}
\end{proof}
The main result of this section is the following extension of theorem 
\ref{th:lof2}:
\begin{theo}\label{th:exactorder}
 If $\tX_1$ is a subspace of order exactly $\theta_0$, the map $\Pi:\ X_1\to X_1/\tX_1$ extends 
 continuously on $X_\theta$ for $\theta>\theta_0$. Denote 
\begin{equation}
\tX_{\theta,p}=
\left\{
\begin{array}{ll}
[X_0,X_1]_{\theta,p},\ \theta < \theta_0,\\
\left[X_0,X_1\right]_{\theta,p}\cap \text{Ker}(\Pi),\ \theta>\theta_0.
\end{array}
\right.
\end{equation}
Then we have for $\theta\neq \theta_0$, $[X_0,\tX_1]_{\theta,p}=\tX_{\theta,p}.$
\end{theo}
\begin{proof}
 We first check that $\Pi$ can be extended : for $u\in X_{\theta,p}$, there exists 
 $(v(t))$ such 
 that 
 $$
 u=\int_0^\infty\frac{v(t)}{t}dt,\ \text{with }\int_0^\infty\left(\frac{J(t,v(t))}{t^\theta}\right)^p
 \frac{dt}{t}<\infty.
 $$
Up to a small modification of $v$, we can replace the integration interval 
$]0,\infty[$ by $]0,1]$, indeed 
$$
\int_1^\infty \frac{\|v(t)\|_{X_1}}{t}dt\leq \int_1^\infty 
 \frac{J(t,v(t))}{t^2}dt\leq \|J(t,v)/t^\theta\|_{L^p(dt/t)}
 \|1/t^{(1-\theta)}\|_{L^{p'}([1,\infty[,dt/t)}<\infty,
$$
so it suffices to modify $v$
as $v(t)1_{[0,1]}+\frac{\int_1^\infty v(s)ds/s}{\ln 2}1_{[1/2,1]}$.\\
Now to check that $\Pi(\int_0^1v(t)dt/t)$
is well defined, we use $\|(v(t)\|_{X_1/\tX_1}\sim t^{-\theta_0}N_J(t,v(t))$ to deduce 
$$
\int_0^1\frac{\|v(t)\|_{X_1/\tX_1}}{t}dt\lesssim \int_0^1\frac{J(t,v(t))}{t^{\theta}}
\frac{dt}{t^{1+\theta_0-\theta}}\leq \|J(t,v(t))/t^{\theta}\|_{L^p(dt/t)}
\|t^{\theta-\theta_0}\|_{L^{p'}(]0,1],dt/t)}<\infty,
$$
this suffices to extend $\Pi$.
\\
Next we turn to the identification of the space $[X_0,\tX_1]_{\theta,p}$: we use the discrete characterization of this space with the K-method : given $u\in X_\theta$, 
for any $j\in \N$, 
$$
\exists\,(u_0^j,u_1^j)\in X_0\times X_1,\ u=u_0^j+u_1^j,\ K(2^{-j},u)\sim \|u_0^j\|_{X_0}+
2^{-j}\|u_1^j\|_{X_1},
$$
$$
\|u\|_{X_{\theta,p}}\sim \left(\sum_0^\infty (\|u_0^j\|_{X_0}+2^{-j}\|u_1^j\|_{X_1})^p2^{p\theta j}
\right)^{1/p}.
$$
Moreover by assumption there exists $(v_1^j)_{j\in \N}$ with $\Pi(v_1^j)=\Pi(u_1^j)$ and 
$$
J(2^{-j},v_1^j)\sim 2^{-j\theta_0}\|u^j_1\|_{X_1/\tX_1}.
$$
\paragraph{The case $\theta<\theta_0$} 
We have $u=u_0^j+v_1^j+u_1^j-v_1^j,\ u_1^j-v_1^j\in \tX_1$, hence the bound
\begin{eqnarray}\nonumber
K(2^{-j},u,X_0,\tX_1)=\tK(2^{-j},u)&\leq& \|u_0^j+v_1^j\|_{X_0}+2^{-j}\|u_1^j-v_1^j\|_{X_1}
\\
\nonumber
&\lesssim& K(2^{-j},u)+J(2^{-j},v_1^j)
\\
\label{eq:tKtoK}
&\sim& K(2^{-j},u)+2^{-j\theta_0}\|v_1^j\|_{X_1/\tX_1}.
\end{eqnarray}
By definition $\|2^{\theta j}K(2^{-j},u)\|_{l^p}=\|u\|_{[X_0,X_1]_{\theta,p}}$, for the second term we use the standard trick involving a telescopic series
\begin{eqnarray*}
v_1^j=v_1^0+\sum_{0}^{j-1}v_1^{k+1}-v_1^{k}\Rightarrow 2^{-j\theta_0}
\|v_1^j\|_{X_1/\tX_1}&\lesssim &2^{-j\theta_0}\|v_1^0\|_{X_1}+
\sum_{0}^{j-1}2^{-j\theta_0}\|v_1^{k+1}-v_1^{k}\|_{X_1/\tX_1}
\\
&\lesssim& 2^{-j\theta_0}\|v_1^0\|_{X_1}+ \sum_{0}^{j-1}2^{(k-j)\theta_0} J(2^{-k},u_1^{k+1}-u_1^{k}).
\end{eqnarray*}
Obviously, $2^{j(\theta-\theta_0)}\|v_1^0\|_{X_1}\in l^p(\N)$, so we focus on the 
control of the sum : note that $u_1^k-u_1^{k+1}=u_0^{k+1}-u_0^{k}$, hence 
$$
J(2^{-k},u_1^{k+1}-u_1^{k})=\|u_0^{k+1}-u_0^k\|_{X_0}+2^{-k}\|u_1^k-u_1^{k+1}\|_{X_1}
\lesssim K(2^{-k},u).
$$
We deduce 
\begin{eqnarray*}
2^{j\theta}\sum_{k=0}^{j-1}2^{(k-j)\theta_0} J(2^{-k},u_1^{k+1}-u_1^{k})&\lesssim & \sum_{0}^{j-1} 2^{(k-j)\theta_0}
2^{(j-k)\theta} 2^{k\theta}K(2^{-k},u)
\\
&=& \sum_{k=0}^j 2^{-(j-k)(\theta_0-\theta)}2^{k\theta} K(2^{-k},u).
\end{eqnarray*}
We recognize a convolution product between $(2^{-j(\theta_0-\theta)}1_{j>0})$ and 
$(2^{j\theta} K(2^{-j},u)1_{j\geq 0})$, hence using Young's inequality and plugging 
the bound into \eqref{eq:tKtoK}
$$
\|2^{j\theta}\tK(2^{-j},u)\|_{l^p}\lesssim \|2^{j\theta}K(2^{-j},u)\|_{l^p}.
$$
This gives  the inclusion $[X_0,X_1]_{\theta,p}\subset [X_0,\tX_1]_{\theta,p}$, and the converse is obvious\footnote{It is interesting to remark the constants in the norm equivalence depend on $\|(2^{-j(\theta_0-\theta)})\|_{l^1}$, and thus degenerate at an 
explicit rate as $\theta\to \theta_0$}. 
\paragraph{The case $\theta >\theta_0$} We assume $u\in X_{\theta,p}\cap \text{Ker}(\Pi)$. 
We restart with the same decomposition as in the case $\theta<\theta_0$,
$u=u_0^j+v_1^j+u_1^j-v_1^j$, as previously the proof boils down to check 
$\|2^{j(\theta-\theta_0)}\|v_1^j\|_{X_1/\tX_1}\|_{l^p(\N)}\lesssim 
\|u\|_{X_{\theta,p}}$. We use the telescopic series
\begin{eqnarray*}
\Pi v_1^j=\sum_{k=j}^\infty \Pi v_1^k-\Pi v_1^{k+1}\Rightarrow 2^{-j\theta_0}
\|v_1^j\|_{X_1/\tX_1}
&\leq& \sum_j^\infty 
2^{-j\theta_0}\|v_1^k-v_1^{k+1}\|_{X_1/\tX_1}
\\
&\sim&\sum_{k=j}^\infty 2^{(k-j)\theta_0}N_J(2^{-k},v_1^k-v_1^{k+1}).
\end{eqnarray*}
Note that it is key here that $u\in \text{Ker}(\Pi)$ to ensure the first equality.
\\
Now as in the case $\theta<\theta_0$, we have $N_J(2^{-k},v_1^k-v_1^{k+1})
=N_J(2^{-k},u_1^k-u_1^{k+1})\lesssim K(2^{-k},u)$,
hence the bound
$$
2^{j(\theta-\theta_0)}\|v_1^j\|_{X_1/\tX_1}\lesssim 
\sum_{k=j}^\infty 
2^{(j-k)(\theta-\theta_0)}2^{k\theta}K(2^{-k},u),
$$
this leads to the same conclusion using $(2^{k(\theta-\theta_0)}1_{k\leq 0})\in l^1(\Z)$:
$$
\|2^{j\theta}\tK(2^{-j},u)\|_{l^p}\lesssim \|2^{j\theta}K(2^{-j},u)\|_{l^p}.
$$
This ensures the inclusion $\tX_{\theta,p}\subset [X_0,\tX_1]_{\theta,p}$, and the inverse 
inclusion is obvious by continuity of the operator $\Pi$ on $X_{\theta,p}$.
\end{proof}

We refer to L\"ofstr\"om \cite{Lofmieux} section $3$ for examples of applications. Unfortunately as we will readily see, this result does not apply directly 
even for the most classical interpolation problem of identifying $[L^2(\R^{d-1}\times \R^{+*}),
H^1_0(\R^{d-1}\times \R^{+*})]_{\theta,2}$ if $d>1$.
\\
When working on the domain $\R^{+*}\times \R^{d-1}$, it is convenient 
to denote $x=(x_1,x')$ and use the tangential Fourier transform 
$$
\widehat{u}(x_1,\xi')=\int_{\R^{d-1}}u(x_1,x')e^{-ix'\cdot \xi'}dx',\ 
(x_1,\xi')\in \R^{+*}\times \R^{d-1}.
$$
\begin{prop}\label{prop:JH1}
Let $X_0=L^2(\R^{+*}\times \R^{d-1})$, $(X_1,\tX_1)=(H^1(\R^{+*}\times \R^{d-1}),H^1_0(\R^{+*}\times \R^{d-1}))$, the infimum
$$
N_J(t,u)=\inf_{v\equiv u[H^1_0]} (\|v\|_2^2+t^2\|v\|_{H^1}^2)^{1/2}
$$ 
is attained by 
$$
P_tu(x_1,\xi')=e^{-x_1/\lambda_t(\xi')}\widehat{tr(u)}(\xi'),\ \lambda_t=\frac{t}
{\sqrt{1+t^2\la \xi'\ra^2}},\ 
\cxip=\sqrt{1+|\xi'|^2},
$$
where $\text{tr}$ is the trace at $x_1=0$.\\
$P_t$ is a projection and we have 
\begin{equation}\label{eq:estimNJ}
N_J(t,u)=J(t,P_tu)\sim t^{1/2}\left(\int_{\R^{d-1}}\sqrt{1+t^2\cxi^2}
|\widehat{tr(u)}|^2d\xi'\right)^{1/2}. 
\end{equation}
The projection $\Pi:\ H^1(\R^{+*}\times \R^{d-1})\to 
H^1/H^1_0(\R^{+*}\times \R^{d-1})$ is not exactly of order $1/2$.
\end{prop}
\begin{proof}
Taking the formula \eqref{eq:estimNJ} for granted, then the fact that 
the projection $H^1\to H^1/H^1_0$ is not of order $1/2$ 
is  straightforward, since obviously we can not have 
$$
\int_{\R^{d-1}}\cxip |\text{tr}(\widehat{u})|^2d\xi'\leq C\int_{\R^{d-1}} \sqrt{1+t^2\cxip^2}|\widehat{tr(u)}|^2d\xi',
$$
with $C$ independent of $u,t$.\\
The existence of a (unique) infimum to
$$
\inf_{u-w\in H^1_0} J^2(t,w)=\inf_{u-w\in H^1_0} \|w\|_2^2+t^2\|w\|_{H^1}^2,
$$
follows from Hilbert projection theorem. In order to compute it, we use its variational characterization: for any $\varphi \in \mathcal{D}(\R^{+*}\times \R^{d-1})$
$$
0=\text{Re}\int_{\R^{d-1}\times \R^+} P_tu\overline{\varphi}+t^2(P_tu\overline{\varphi}+\n P_tu\cdot \overline{\n\varphi})dx
=\text{Re}\int_{\R^{d-1}\times \R^+}(1+t^2-t^2\Delta)P_t u\overline{\varphi}dx.
$$
Using the Fourier transform this reduces to
$$
\partial_1^2\widehat{P_tu}=\frac{1+t^2\cxip^2}{t^2}\widehat{P_tu}:=\frac{1}{\lambda_t^2}\widehat{P_t u},
$$
hence
$$
\widehat{P_tu}=e^{-x_1/\lambda_t}\widehat{tr(u)}.
$$
The estimate of $J(t,P_tu)$ is then a direct computation : 
$$
\|P_tu\|_2^2=\int_{\R^{d-1}} \frac{\lambda_t}{2}|\widehat{tr(u)}|^2d\xi',
$$
\begin{eqnarray*}
\|P_tu\|_{H^1}^2=\int_{\R^{d-1}}\left(\frac{1}{2\lambda_t}+\frac{\lambda_t}{2}\cxip^2\right)|\widehat{tr(u)}|^2d\xi'
&\sim &\int_{\R^{d-1}}\frac{1+t^2\cxip^2+t^2\cxip^2}{t\sqrt{1+t^2\cxip^2}}|\widehat{tr(u)}|^2d\xi'
\\
&\sim& \int_{\R^{d-1}}\frac{\sqrt{1+t^2\cxip^2}}{t}|\widehat{tr(u)}|^2d\xi'.
\end{eqnarray*}
Hence as expected 
\begin{eqnarray*}
J(t,P_tu)^2&\sim& t\int_{\R^{d-1}}\left(\frac{1}{\sqrt{1+t^2\cxip^2}}+\sqrt{1+t^2\cxip^2}\right)
|\widehat{tr(u)}|^2d\xi'
\\
&\sim &t\int_{\R^{d-1}}\sqrt{1+t^2\cxip^2}|\widehat{tr(u)}|^2d\xi'.
\end{eqnarray*}
\end{proof}

\section{Order for linearisable subspaces and applications}\label{sec:linorder}
The aim of this section is to give a refinement of theorem \ref{th:exactorder} in 
the linearisable case, which is easily applicable to, among other things, the case 
from proposition \ref{prop:JH1}.\\
We start with a simple abstract result which seems new and of independent interest : 
\begin{prop}\label{prop:abstract}
 If $(X_0,X_1)$ is linearisable with linearisation operators $(\Lambda(t))_{0<t\leq 1}$,
 then we have the following norm equivalence on $[X_0,\tX_1]_{\theta,p}$ 
 $$
 \|u\|_{[X_0,\tX_1]_{\theta,p}}\sim \|u\|_{X_{\theta,p}}+
 \int_0^1\left(\frac{N_J(t,\Lambda(t)u)}{t^{\theta}}\right)^p\frac{dt}{t}.
 $$
\end{prop}
\begin{proof}
 Since $\tK\geq K$, it suffices to prove 
$\tK\lesssim K(t,u)+N_J(t,\Lambda(t)u)$ and $N_J(t,\Lambda(t)u)\lesssim \tK(t,u)$.\\
For $u\in X_{\theta,p}$, we choose $v(t)$ in the class of $\Lambda(t)u$ modulo
$\tX_1$  
such that $J(t,v(t))\sim N_J(t,\Lambda(t)u)$ and 
split
$$
u=u-\Lambda(t)u+v(t)+\Lambda(t) u-v(t),
$$
so the first bound is immediate
$$
\tK(t,u)\lesssim \|u-\Lambda(t)u\|_{X_0}+\|v(t)\|_{X_0}+t\|\Lambda(t)u\|_{X_1}
+t\|v(t)\|_{X_1}
\lesssim K(t,u)+J(t,v(t)).
$$
For the second bound, we decompose 
$$u=u-y(t)+y(t),\ y(t)\in \tX_1\text{ and }
\tK(t,u)\sim \|u-y(t)\|_{X_0}+t\|y(t)\|_{X_1}.
$$
Then we have since $y(t)\in \tX_1$, $N_J(t,y(t))=0$, so using \eqref{boundLambda}
\begin{eqnarray*}
N_J(t,\Lambda(t)u)&=&N_J(t,\Lambda(t)(u-y(t))+(\Lambda-I)y)
\\
&\leq& J(t,\Lambda(t)(u-y(t)))+J(t,\Lambda(t) y(t)-y(t))
\\
&\leq& \|u-y(t)\|_{X_0}+t\|y(t)\|_{X_1}\sim \tK(t,u).
\end{eqnarray*}
\end{proof}

When the couple $(X_0,X_1)$ is linearisable, we define a slightly weaker 
version of ``order''.
\begin{define}Let $(X_0,X_1)$ linearisable, $\Lambda$ a linearisation operator.
We say that a closed subspace $\tX_1\subset X_1$ has interpolation order
$\theta_0$ if $X_1/\tX_1$ is embedded\footnote{By embedded, we do not 
necessarily mean inclusion, only the existence of linear injective mapping, and abusively denote $\|u\|_Y$ without mention of the mapping.} in a Banach space $(Y,\|\cdot\|_Y)$ such that 
for any $0<s,t\leq 1,$ $t\sim s$, $u,v\in (X_0)^2$
\begin{equation}\label{eq:equivsubtle}
N_J(t,\Lambda(t)u+\Lambda(s)v)\sim t^{\theta_0}\|\Lambda(t)u+\Lambda(s)v\|_Y,
\end{equation}
 \end{define}

 In these settings, we have the following result, which is both an extension and a 
refinement of theorem \ref{th:exactorder}:
\begin{theo}\label{th:interplin}
Let a couple $(X_0,X_1)$ linearisable with linearisation 
operator $\Lambda$, $\tX_1$ a closed subspace of $X_1$ with interpolation order $\theta_0$ associated to a space $Y$.
Then we have 
$$
[X_0,\tX_1]_{\theta,p}=
\left\{u:\ \|u\|_{[X_0,X_1]_{\theta,p}}+\big\|\|\Lambda(t)u\|_Yt^{\theta_0-\theta}\big\|
_{L^p([0,1],dt/t)}:=\|u\|_{[X_0,\tX_1]_{\theta,p}}<\infty\right\},
$$
and the norm on $[X_0,\tX_1]_{\theta,p}$ is equivalent to $\|\cdot\|_{[X_0,\tX_1]_{\theta,p}}$. 
\\
Moreover, the projection $\Pi:\ X_1\to X_1/\tX_1\hookrightarrow  Y$ extends continuously 
$[X_0,X_1]_{\theta,p}\rightarrow Y$ for $\theta>\theta_0$ and except for the critical case $\theta=\theta_0$, 
the interpolation spaces are the following closed subspaces of $[X_0,X_1]_{\theta,p}$ :
\begin{equation}
[X_0,\tX_1]_{\theta,p}=
\left\{
\begin{array}{ll}
[X_0,X_1]_{\theta,p},\ \theta < \theta_0,\\
\left[X_0,X_1\right]_{\theta,p}\cap \text{Ker}(\Pi),\ \theta>\theta_0.
\end{array}
\right.
\end{equation}

\end{theo}

\begin{proof}
The norm equivalence is a direct application of proposition \ref{prop:abstract}.\\
To extend the projection, note that since $X_1\subset X_0$, we may assume 
$\Lambda(2^{0})=0$, and decompose $u\in X_{\theta,p}$ as 
$$
u=\sum_0^\infty \Lambda(2^{-j-1},u)-\Lambda(2^{-j},u).
$$
Then 
$$
\|\Lambda(2^{-j-1})u-\Lambda(2^{-j})u\|_Y\lesssim 2^{j\theta_0} 
J(2^{-j},\Lambda(2^{-j-1})u-\Lambda(2^{-j})u)\lesssim 2^{j\theta_0}K(2^{-j},u),
$$
hence if $\theta>\theta_0$, 
$\di \sum_0^\infty \|\Lambda(2^{-j-1})u-\Lambda(2^{-j})u\|_Y$ converges.
\vspace{4mm}\\
We deduce now the identification in the special cases $\theta>\theta_0$, $\theta<\theta_0$. 
The inclusion $[X_0,X\tX_1]_{\theta,p}\subset [X_0,X_1]_{\theta,p}$ 
is clear. For $\theta>\theta_0$ the 
inclusion $[X_0,\tX_1]_{\theta,p}\subset [X_0,X_1]_{\theta,p}\cap \text{Ker}(\Pi)$ follows
from the fact that $\Pi$ is well defined on 
$X_{\theta,p}$ : indeed we can write any $u\in [X_0,\tX_1]_{\theta,p}$ as 
$u=\sum_0^\infty v_j$ with convergence in $Y$, and each $v_j\in \text{Ker}(\Pi)$, so by continuity $u\in \text{Ker}(\Pi)$.
\\
First we reformulate the integral in a discrete version :
\begin{eqnarray*}
\int_0^1
\left(\frac{\|\Lambda(t))u\|_Y}{t^{\theta-\theta_0}}\right)^p\frac{dt}{t}
&\sim &\sum_0^\infty \|\Lambda(2^{-j})u\|_Y^p2^{jp(\theta-\theta_0)}
\\
&=&\|\|\Lambda(2^{-j})u\|_Y 2^{j(\theta-\theta_0)}\|_{l^p}.
\end{eqnarray*}
In the case  $\theta>\theta_0$,  assume $u\in X_{\theta,p}\cap Ker(\Pi)$, so that (in $Y$)
$$\Lambda(2^{-j})u=\sum_{k\geq j} (\Lambda(2^{-k})-\Lambda(2^{-k-1}))u,$$
this gives the bound
\begin{eqnarray}
\label{ineq:telesc}
\|\Lambda(2^{-j})\|_Y&\leq& \sum_{k\geq j} \|(\Lambda(2^{-k})-\Lambda(2^{-(k+1)})u\|_Y
\\
\nonumber&\lesssim& \sum_{k\geq j} 2^{k\theta_0}J(2^{-k},(\Lambda(2^{-k})-\Lambda(2^{-(k+1)})u)
\\
\nonumber
&\sim& \sum_{k\geq j} 2^{k(\theta_0-\theta)} 2^{k\theta}K(2^{-k},u).
\end{eqnarray}
Young's inequality implies :
$$
\|\|\Lambda(2^{-j})u\|_Y 2^{j(\theta-\theta_0)}\|_{l^p}
\lesssim \|2^{(k-j)(\theta_0-\theta)}\|_{l^1}\|2^{k\theta}K(2^{-k},u)\|_p\sim 
\|u\|_{X_{\theta,p}},
$$
which gives $u\in [X_0,\tX_1]_{\theta,p}$.
\\
The proof for $\theta<\theta_0$ is similar, the only difference being that 
at step \eqref{ineq:telesc} we decompose (as in the proof of theorem \ref{th:exactorder})
$\Lambda(2^{-j})u=\sum_{k=0}^{j-1} \left(\Lambda(2^{-(k+1)})-\Lambda(2^{-k})\right)u$, 
details are left to the reader.
\end{proof}
\paragraph{Example 1 : $[L^2,H^1_0]$} 
We come back to the standard example $X_0=L^2,\ X_1=H^1(\R^{+*}\times \R^{d-1}),\  
\tX_1=H^1_0$. The natural candidate for the $Y$ space is 
$L^2(\R^{d-1})$, and we naturally embed $X_1/\tX_1$ into $Y$ thanks to the trace 
operator.
\\
We will show how to rapidly derive the classical result 
$$
\|u\|_{H^{1/2}_{00}}\sim \|u\|_{H^{1/2}}+\|u/x_1^{1/2}\|_{L^2}.
$$
We recall the notation $x=(x_1,x')$, and similarly $\xi=(\xi_1,\xi')$, 
we denote $\hat{\cdot}$ the Fourier transform on $L^2(\R^d)$ and $\cF'$ 
the one on $L^2(\R^{d-1})$.\\
For $\Lambda(t)$, we take 
$$
\Lambda(t)u=\cF^{-1}\left(1_{|\xi'|\leq 1/t}\frac{1}{1+t^2\xi_1^2}\widehat{Eu}
\right)
,\ Eu
\text{ even extension on }\R^d.
$$
It is easily seen that $\Lambda$ linearizes the couple $(L^2,H^1)$ : for $0\leq t\leq 1$,
\begin{eqnarray*}
t^2\|\Lambda(t)u\|_{H^1}^2=\int_{\R^d}1_{|\xi'|\leq 1}\frac{t^2\cxi^2}{(1+t^2\xi_1^2)^2}|\widehat{Eu}|^2d\xi
&\lesssim& \int_{\R^d}\frac{t^2\la\xi_1\ra^2}{(1+t^2\xi_1^2)^2}|\widehat{Eu}|^2d\xi
\\
&\lesssim& \min(t^2\|u\|_{H^1}^2,\|u\|_2^2),
\end{eqnarray*}
\begin{eqnarray*}
\|(1-\Lambda(t))u\|_2^2&=&\int_{\R^d}\left(\frac{1+t^2\xi_1^2
-1_{|\xi'|\leq 1/t}}{1+t^2 \xi_1^2}\right)^2
|\widehat{Eu}|^2d\xi
\\
&=&\int_{\R^d}\left(\frac{1_{|\xi'|\geq 1/t}+t^2\xi_1^2}
{1+t^2 \xi_1^2}\right)^2
|\widehat{Eu}|^2d\xi
\\
&\leq& \min(t^2\|u\|_{H^1}^2,\|u\|_2^2).
\end{eqnarray*}
From proposition \ref{prop:JH1} we have 
$$
N_J(t,\Lambda(t)u)\sim t^{1/2}\left(\int_{\R^{d-1}}\sqrt{1+t^2|\xi'|^2}|\widehat{\text{tr}(\Lambda(t)u)}|^2\,d\xi'\right).
$$
Since $\cF'(\text{tr}(\Lambda(t)u))(\xi')=
\frac{1}{2\pi}\int_\R\widehat{\Lambda(t)u}(\xi_1,\xi')d\xi_1$, we have that
$\widehat{\text{tr}(\Lambda(t)u)}$ is supported in $|\xi'|\lesssim 1/t$, hence
we see that condition \eqref{eq:equivsubtle} is satisfied
$$
N_J(t,\Lambda(t)u)\sim t^{1/2}\left(\int_{\R^{d-1}}|\widehat{\text{tr}(\Lambda(t)u)}|^2\,d\xi'\right)=t^{1/2}\|\Lambda(t)u\|_{Y}.
$$
The identification of $[L^2,H^1_0]_{1/2,2}$ is then derived : denote
$v_{<1/t}(x_1,x')$ the function defined by 
$$
\cF'(v_{1/t})(x_1,\xi'))=
1_{|\xi'|<1/t}\cF'(v)(x_1,\xi').$$
Using $\cF^{-1}(1/(1+\xi_1^2))=e^{-|x_1|}/2$, we have
$$
2\text{tr}(\Lambda(t)u)(x')=\int_0^\infty u_{<1/t}(x_1,x')\frac{e^{-|x_1|/t}}{t}\,dx_1,
$$
then 
\begin{eqnarray*}
\|2\text{tr}(\Lambda(t)u)\|_{L^2(dx'dt/t)}^2&=&\int_0^1\int_{\R^{d-1}}
\int_0^\infty\int_0^\infty u_{<1/t}(y,x')\overline{u_{<1/t}}(z,x')dydzdx'e^{-(y+z)/t}\frac{dt}{t^3}
\\
&\leq&\int_0^1\int_0^\infty\int_0^\infty \frac{\|u_{<1/t}(y,\cdot)\|_2
\|u_{<1/t}(z,\cdot)\|_2}{t^3}e^{-(y+z)/t}dt
\\
&\leq& \int_0^1\int_0^\infty\int_0^\infty \frac{\|u(y,\cdot)\|_2
\|u(z,\cdot)\|_2}{t^3}e^{-(y+z)/t}\,dydzdt
\\
&=&\int_0^\infty\int_0^\infty \frac{\|u(y,\cdot)\|_2
\|u(z,\cdot)\|_2}{(y+z)^2}\,dydz
\end{eqnarray*}
Splitting the integral as $z<y$ and $y<z$, and using the fractional Hardy inequality 
we bound
\begin{eqnarray*}
\int_0^\infty\frac{\|u(y,\cdot)\|_2}{y^{1/2}}
\frac{1}{y^{3/2}}
\int_0^y\frac{\|u(z,\cdot)\|_2}{(1+z/y)^2}\,dzdy
&\leq& \|\|u(y,\cdot)\|_2/y^{1/2}\|_{L^2(\R^+)}
\left\|\frac{1}{y^{3/2}}\int_0^y\|u(z,\cdot)\|_2dz\right\|_2
\\
&\lesssim& \|\|u(y,\cdot)\|_2/y^{1/2}\|_{L^2(\R^+)}^2.
\end{eqnarray*}
This proves that the $[L^2,H^1_0]_{1/2,2}$ norm is controlled 
by $\|u\|_{H^{1/2}}+\|u/x_1^{1/2}\|_2$, the reverse inequality follows from the
usual Hardy inequality $\|u/x_1\|_{L^2}\lesssim \|\partial_1 u\|_{L^2}$ and 
an interpolation argument.\\
For completeness, we recall that the fractional Hardy inequality is a consequence 
of Minkowski's inequality
\begin{eqnarray*}
\|y^{-3/2}\int_0^yf(z)dz\|_2&=&\|y^{-1/2}\int_0^1f(ys)ds\|_2
\\
&\leq &\int_0^1\left(\int_0^\infty\frac{f^2(ys)}{y}dy\right)^{1/2}ds
\\
&=&\int_0^1\left(\int_0^\infty\frac{f^2(x)}{x}dx\right)^{1/2}ds=\|f/x^{1/2}\|_2.
\end{eqnarray*}
\paragraph{Example 2 : anisotropic Sobolev spaces}
Take $X_0=H^{1/2,2}(\R^{+*}_{s}\times\R^{d-1})\times L^2(\R^{+*}_{x_1}\times\R^{d-1})$, $X_1=H^{3/2,2}(\R^{+*}_{s}\times\R^{d-1})\times H^1(\R^{+*}_{x_1}\times\R^{d-1})$, where 
$$
\|u\|_{H^{\theta,2}(\R^d)}^2=\int_{\R^d}(1+|\xi|^2+|\delta|)^\theta|\widehat{u}(\xi,\delta)|^2
\,d\xi\,d\delta,
$$
and functions in $H^{\theta,2}(\R^{+*}_s\times \R^{d-1})$ are restrictions of functions 
in $H^{\theta,2}(\R^{d})$.\\
We set $\tX_1=\{(g,u)\in X_1:\ u|_{x_1=0}-g|_{s=0}=0\}$. 
\\
It is convenient to distinguish the domains $\R^{+*}_{x_1}\times\R^{d-1}$
and $\R^{+*}_s\times\R^{d-1}$, to do so we denote (as indicated by the notation) 
their respective first variable $x_1$ and $s$.
\\
We shall prove here that if 
\begin{equation}\label{eq:interpaniso}
(g,u)\in H^{1,2}\times H^{1/2}\text{ and } \int_0^1\|u(t,x')-g(t^2,x')\|_{L^2(dx')}^2
\frac{dt}{t}<\infty,\text{ then }(g,u)\in [X_0,\tX_1]_{1/2,2}.
\end{equation}
We take $(Y,\|\cdot\|_Y)=(L^2(\R^{d-1}),\|\cdot\|_2)$, and the embedding 
$X_1/\tX_1$ into $Y$ is given by the map $(g,u)\to (g|_{s=0}-u|_{x_1=0})$.
A reminder about trace in anisotropic Sobolev spaces is given just after the next 
paragraph.
\\
Interpolation problems with anisotropic Sobolev spaces appear naturally in the analysis of 
initial boundary value problems where time differentiation is of order one, while 
spatial derivatives appear at order two, eg parabolic partial differential equations 
(see Amann \cite{Amann2}). This specific example was considered by the author 
for the Schr\"odinger equation on a domain \cite{Audiard6}, where the 
critical case of identifying $[X_0,\tX_1]_{1/2,2}$ turned out to be essential.\\
Define $X_\theta=H^{\theta+1/2,2}\times H^\theta$, the trace map $(g,u_0)\to (g|_{s=0},\ u_0|_{x_1=0})$
sends $X_\theta$ to $H^{\theta-1/2}\times H^{\theta-1/2}$ for $\theta>1/2$. More precisely, 
we have the standard trace estimate
\begin{equation}
\label{eq:anisotrace}
\forall\,g\in H^{3/2,2},\ \|g|_{t=0}\|_2\lesssim \|g\|_{H^{1/2,2}}^{1/2}\|g\|_{H^{3/2,2}}^{1/2}.
\end{equation}
\begin{equation}
\label{eq:isotrace}
\forall\,u\in H^{1},\ \|u_0|_{x_1=0}\|_2\lesssim \|u_0\|_{L^2}^{1/2}\|u_0\|_{H^{1}}^{1/2}.
\end{equation}
This implies that the infimum $N_J(t,(g,u))=\inf_{(h,v)\in \tX_1} J(t,(g-h,u-v))$, is bounded from below by 
\begin{equation*}
N_J(t,(g,u))\gtrsim t^{1/2}\inf_{(h,v)\in \tX_1}\left(\|(g-h)|_{s=0}\|_{L^2}+\|(u-v)|_{x_1=0}\|_{L^2}\right).
\end{equation*}
On the other hand the infimum
$$
\inf_{\varphi} \int_{\R^{d-1}}|\widehat{\text{tr}(u)}-\widehat{\varphi}|^2+
|\widehat{\text{tr}(g)}-\widehat{\varphi}|^2\,d\xi,
$$
is attained for $\varphi=(\text{tr}(u)+\text{tr}(g))/2$, hence 
\begin{equation}
\label{order:upper}
t^{1/2}\|\text{tr}(u)-\text{tr}(g)\|_2\lesssim N_J(t,u).
\end{equation}
To check the opposite inequality, we define a lifting 
$R:\ H^{1/2}(\R^{d-1})\to H^{3/2,2}(\R_s\times \R^{d-1})$ with some parameter $\lambda$ to choose later
$$
\cF_{x'}R\varphi(\xi',s)=e^{-\lambda |s|}\widehat{\varphi}(\xi'), \ 
\widehat{R\varphi}(\xi',\delta)=\frac{1}{\lambda(1+(\delta/\lambda)^2)}
\widehat{\varphi}(\xi').
$$
The $H^{s,2}$ norm of $R\varphi$ is 
\begin{eqnarray*}
\iint_{\R^d} (1+|\xi'|^2+|\delta|)^s\frac{1}{\lambda^2(1+(\delta/\lambda)^2)^2}
|\widehat{\varphi}(\xi')|^2\, d\xi'\, d\delta
=\iint_{\R^d} (1+|\xi'|^2+\lambda |\delta|)^s&\di \frac{1}{\lambda(1+(\delta)^2)^2}
\\
&\times|\widehat{\varphi}(\xi')|^2\, d\xi'\, d\delta,
\end{eqnarray*}
so 
$$
\|R\varphi\|_{H^{1/2,2}}^2+t^2\|R\varphi\|_{H^{3/2,2}}=
\iint_{\R^d} \frac{\left((1+|\xi'|^2+\lambda |\delta|)^{1/2}+t^2(1+|\xi'|^2+\lambda|\delta|)^{3/2}
\right)}{\lambda(1+\delta^2)^2}|\widehat{\varphi}(\xi')|^2\,d\xi'\,d\delta.
$$
Take $\lambda=(1+t^2|\xi'|^2)/t^2$, we get the bound 
\begin{eqnarray}\nonumber
J(t,R\varphi)^2&\lesssim &\iint_{\R^d} \frac{((1+t^2|\xi'|^2)(1+|\delta|))^{1/2}+t^2((1+t^2|\xi'|^2)(1+|\delta|))^{3/2}}{(1+t^2|\xi'|^2)(1+\delta^2)^2}|\widehat{\varphi}|^2 t^2\,d\xi'\,d\delta
\\
\label{order:lower}
&\lesssim& t\iint_{\R^d}\frac{((1+t^2|\xi'|^2)(1+|\delta|))^{3/2}}{(1+t^2|\xi'|^2)(1+\delta^2)^2}|\widehat{\varphi|^2}d\xi'\,d\delta
\\
&\sim& t \int(1+t^2|\xi'|^2)|\widehat{\varphi}|^2d\xi'.
\end{eqnarray}
As for the homogeneous case, using 
$$
\Lambda(t)(g,u)=(\cF')^{-1}\left(1_{|\xi'|<1/t}\frac{\widehat{g}(\delta,\xi')}{1+t^4\delta^2},
1_{|\xi'|<1/t}\frac{\widehat{u}(\xi)}{1+t^2\xi_1^2}\right),
$$
the bounds \eqref{order:upper},\eqref{order:lower}, and proposition \ref{prop:JH1},
we deduce that we have
$$
N_J(t,\Lambda(t)(u,g))\sim t^{1/2}\|\Lambda(t)(u,g)\|_Y,
$$
namely \eqref{eq:equivsubtle} with $\theta_0=1/2$.
\\
We use as in the previous example the subscript $\widehat{u}_{<1/t}=1_{|\xi'|\leq 1/t}
\widehat{u}$.  
To evaluate the norm $N_J$, take the difference of the traces, 
$$
\int_{\R^+}u_{<t}(x)\frac{e^{-x_1/t}}{t}dx_1-\int_{\R^+}g_{<t}(s,x')\frac{e^{-s/t^2}}{t^2}\,ds
=\int_{\R^+}(u_{<t}(yt,x')-g_{<t}(yt^2,x'))e^{-y}dy.
$$
Take the $L^2(dx')$ norm : thanks to Minkowski's inequality and the boundedness 
of $1_{|\xi'|<1/t}$, we are reduced to bound the $L^2(dt/t)$ norm of 
$\int_{\R^+}e^{-y}\|u(yt,\cdot)-g(yt^2,\cdot)\|_{L^2(dx')}dy$. In the remainder of the proof, 
for conciseness, we shall drop this norm and simply write $|u(yt)-g(yt^2)|$ 
instead of $\|u(yt,\cdot)-g(yt^2,\cdot)\|_{L^2(dx')}$. With these notations :
\begin{eqnarray*}
\|N_J(t,\Lambda(t)(g,u))\|_{L^2(dt/t)}^2&\leq &
\int_0^1\int_{\R^+}\int_{\R^+}|u(yt)-g(yt^2)|e^{-(x+y)}
|u(xt)-g(xt^2)|dydx\frac{dt}{t}
\\
&\leq &\iint_{(\R^+)^2}\|u(yt)-g(yt^2)\|_{L^2(dt/t)}
\|u(xt)-g(xt^2)\|_{L^2(dt/t)}e^{-(x+y)}dxdy.
\end{eqnarray*}
We can evaluate 
\begin{eqnarray*}
\|u(yt)-g(yt^2)\|_{L^2(dt/t)}&=&\|u(s)-g(s^2/y)\|_{L^2(ds/s)}
\\
&\leq &\|u(s)-g(s^2)\|_{L^2(ds/s)}+\|g(s^2)-g(s^2/y)\|_{L^2(ds/s)}.
\end{eqnarray*}
The property \eqref{eq:interpaniso} is true if  the second term is finite for 
$g\in H^{1,2}$. 
Actually we have the more precise estimate
\begin{equation}\label{eq:bornedilatation}
\int_0^\infty\int_0^\infty (g(s^2)-g(s^2/y))^2\frac{ds}{s}e^{-y}dy
\lesssim \|g\|_{H^{1/2}(\R_t^+)}^2.
\end{equation}
To check this, first the change of variable $x=s^2$ reduces to prove 
$$
\int_0^\infty\int_0^\infty (g(x)-g(x/y))^2\frac{dx}{x}e^{-y}dy\lesssim \|g\|_{H^{1/2}}^2.
$$ 
To do so we
use an interpolation argument between $L^2$ and $H^1$, indeed from the triangular 
inequality
\begin{equation}\label{eq:dilat0}
\int_0^\infty\int_0^\infty (g(x)-g(x/y))^2\,dx\,e^{-y}dy\lesssim \int_0^\infty 
e^{-y}(\|g\|_2^2+y\|g\|_2^2)dy\lesssim \|g\|_2^2.
\end{equation}
On the other hand consider $\di \int_0^\infty \int_0^\infty \frac{(g(x)-g(x/y))^2}{x^2}dx\,e^{-y}dy$ : the integral over $x\in [1,\infty[$ is easy to bound, the integral on 
$[0,1]$ is bounded using Taylor's formula and Minkowski's inequality as follows :
\begin{eqnarray}\nonumber
 \int_0^\infty \int_0^1\frac{(g(x)-g(x/y))^2}{x^2}dx\,e^{-y}dy
 &=&\int_0^\infty \int_0^1 \frac{(\int_0^1g'(x+tx(1/y-1)x(\frac{1}{y}-1)dt)^2}{x^2}dx\,e^{-y}dy
 \\
 \nonumber
 &=&\left\|\int_0^1g'(x+tx(1/y-1))|^2
 \left(\frac{1}{y}-1\right) \right\|_{L^2(e^{-y}dydx)}
 \\
 \nonumber
 &\leq& \left(\int_0^1\left(\int_0^\infty\int_0^1 |g'(x+tx(1/y-1))|^2
 \left(\frac{1}{y}-1\right)^2dxe^{-y}dy\right)^{1/2}dt\right)^2
 \\
 \nonumber
 &=&\left(\int_0^1\left(\int_0^1\int_0^\infty |g'(x+tz)|^2\left(\frac{z}{1+z}\right)^2e^{-1/(1+z)}dzdx\right)^{1/2}dt\right)^2
 \\
 \nonumber
 &\lesssim& \left(\int_0^1\left(\int_0^1\frac{\|g'\|_2^2}{t}dx\right)^{1/2}dt\right)^{2}
 \\
 &\lesssim& \|g'\|_2^2.
 \label{eq:dilat1}
\end{eqnarray}
Combining \eqref{eq:dilat0},\eqref{eq:dilat1} we obtain 
\eqref{eq:bornedilatation} by interpolation. This implies the required bound 
$$
\|N_J(t,\Lambda(t)(g,u)\|_{L^2(dt/t)}\lesssim \|u(s,x')-g(s^2,x')\|_{L^2(dx'ds/s)}
+\|(g,u)\|_{H^{1,2}\times H^{1/2}}.
$$

\paragraph{Example 3: Codimension $2$ trace} Take $X_0=L^2(\R^d)$, 
$X_1=H^2(\R^d)$, $d\geq 2$, $\tX_1=\{u\in H^2:\ u(0,0,x')=0,\ x'\in \R^{d-2}\}$
(this problem appears in anisotropic settings in \cite{Audiard6}, and the 
result presented here was only conjectured). 
\\
We will prove a property very similar to the first example: 
\begin{equation}\label{interp:point}
u\in H^1\text{ and }\int_{\R^{2}}\frac{\|u(x_1,x_2,\cdot)\|_{L^2(\R^{d-2})}}
{x_1^2+x_2^2}dx_1dx_2\Rightarrow u\in [X_0,\tX_1]_{1/2}. 
\end{equation}
With similar arguments as before, it can be checked that $\tX_1$ has interpolation 
order $1/2$ with $Y=L^2(\R^{d-2})$ (and $\|u\|_{Y}=\|u(0,0,\cdot)\|_{L^2(\R^{d-2})}$),
and linearisation operator
$$
\Lambda(t)u=\frac{e^{-(|x_1|+|x_2|)/t)}}{t^2}*u_{<1/t},\
\cF'u_{<1/t}=\cF'(u)(x_1,x_2,\xi')1_{|\xi'|<1/t}.
$$
We have
\begin{eqnarray*}
\|\Lambda(t)u(0,\cdot)\|_{L^2(dt/t)}^2=\int_0^1\int_{\R^{d-2}}\int_{\R^2}\int_{\R^2}
\frac{1}{t^4}e^{-(|x_1|+|x_2|+|y_1|+|y_2|)/t}u(x_1,x_2,x')u(y_1,y_2,x')\\
\,dx_1\,dx_2\,dy_1\,dy_2\,dx'dt.
\end{eqnarray*}
Skipping computations similar to the previous examples, we find eventually 
the bound
\begin{eqnarray*}
\|\Lambda(t)u(0)\|_{L^2(dt/t)}^2\lesssim 
\int_{\R^4}\frac{\|u(x_1,x_2,\cdot)\|_{L^2(\R^{d-2}}\|u(y_1,y_2,\cdot)\|_2}{(|(x_1,x_2)|
+|(y_1,y_2)|)^4}dx_1dx_2dy_1dy_2,
\end{eqnarray*}
with of course $|(x_1,x_2)|=\sqrt{x_1^2+x_2^2}$.\\
For conciseness, let us denote $\|u(x_1,x_2,\cdot)\|_{L^2(\R^{d-2}}=|u(x)|$, and 
$|x|=\sqrt{x_1^2+x_2^2}$. Splitting 
the integral in two domains $\{|y|<|x|\}\cup \{|y|\geq |x|\}$, and by symmetry, 
we have the bound
\begin{eqnarray*}
\|\Lambda(t)u(0)\|_{L^2(dt/t)}^2\lesssim 
 \int_{\R^2}\frac{|u(x)|}{|x|}\frac{1}{|x|^3}\int_{|y|<|x|}|u(y)|dydx,
\end{eqnarray*}
hence \eqref{interp:point} boils down to the Hardy type inequality 
\begin{equation}
\left\| \frac{1}{|x|^3}\int_{|y|<|x|}|u(y)|dydx\right\|_2\lesssim 
\|u/|x|\|_2.
\end{equation}
To prove the latter, it is simplest to use polar coordinates in complex form 
$x=re^{i\varphi}$, $r>0,\ \varphi\in [0,2\pi[$,
\begin{eqnarray*}
\left\| \frac{1}{|x|^3}\int_{|y|<|x|}|u(y)|dy\right\|_2&\leq&
\left\| \frac{1}{|x|}\int_{B(0,1)}|u(|x|z)|dz\right\|_2
\\
&\leq& \int_{B(0,1)}\left(\int_{\R^2}\frac{u^2(|x|z)}{|x|^2}dx\right)^{1/2}dz
\\
&=&\int_{B(0,1)}\left(2\pi \int_0^\infty\frac{u^2(rz)}{r}dr\right)^{1/2}dz
\\
&=&\int_0^1\int_0^{2\pi}\left(2\pi\int_0^\infty \frac{u^2(rse^{i\varphi})}{r}dr\right)^{1/2}s\,ds\,d\varphi
\\
&\lesssim& \int_0^{2\pi}\left(\int_0^\infty\frac{u^2(\rho e^{i\varphi})}{\rho}d\rho\right)^{1/2}d\varphi
\\
&\lesssim& \left(\int_0^{2\pi}\int_0^\infty\frac{u^2(\rho e^{i\varphi})}{\rho}\,d\rho\,d\varphi\right)^{1/2}
=\|u/|x|\|_{L^2}.
\end{eqnarray*}
\section{Intersection of independent subspaces}\label{sec:indep}
We consider here the case where $\tX_1\subset X_1$ is an intersection $\di \cap_{k=1}^n X_1^k$,
with $X_1^k$ closed subspaces of $X_1$. This part is a quite straightforward 
extension of the notion of ``independent subspace'' from L\"ofstr\"om \cite{LofsSub}, that we 
detail for the convenience of the reader :
\begin{define}
Let $(X_1^k)_{1\leq k\leq n}$ with respective interpolation orders $\theta_1<\theta_2\cdots 
<\theta_n$, associated to spaces $Y_1,\cdots,Y_n$. They are independent when for any $1\leq k\leq n$, 
$$
N_J(t,\Lambda(t)u+\Lambda(s) v)\sim \sum_{k=1}^nt^{\theta_k}\|\Lambda(t)u+\Lambda(s)v\|_{Y_k}.
$$
\end{define}
\begin{rmq}
Our definition of independence does not look the same as in \cite{LofsSub}, but it is 
very close to be an infinite dimensional variant of the equivalent notion of 
``supporting sequence'' 
(definition $2$ in \cite{LofsSub}) : for spaces defined as kernels $X_1^k=\text{Ker}(\Gamma_k)$, 
a supporting sequence is a family $(w_k)_{1\leq k\leq n}$
such that 
$$
\Gamma_j(w_k)=\delta_k^j\text{ and }J(t,w_k)\lesssim t^{\theta_k}.
$$
\end{rmq}

\begin{theo}
Let $\tX_1=\cap_{k=1}^nX_1^k$, where $(X_1^k)_{1\leq k\leq n}$ are independent with 
respective interpolation orders $(\theta_k)_{1\leq k\leq n}$.
 \\
 For $\theta\notin \{\theta_j,\ 1\leq j\leq n\}$, let $k$ the largest integer 
 such that $\theta_k<\theta$ : for $j\leq k$, 
 the projection $\Pi_j:\ X_1\to X_1/X_1^j$ extends 
 $X_{\theta,p}\to Y_j$, and we have 
 $$
[ X_0,\tX_1]_{\theta,p}=X_{\theta,p}\cap \left(\bigcap_{j\leq k}\text{Ker}
\, \Pi_j\right).
 $$
 If there exists  $k$ such that $\theta=\theta_k$, then 
 $$
 [ X_0,\tX_1]_{\theta,p}=\left\{u\in X_{\theta,p}\cap \left(\bigcap_{j<k}\text{Ker}(\Pi_j)\right):\ 
 \|\|\Lambda(t)u\|_{Y_k}\|_{L^p(dt/t)}<\infty\right\}.
 $$
\end{theo}
\begin{proof}
This is a direct modification of the proof of theorem \ref{th:interplin}: 
for any $j$ such that $\theta>\theta_j$, we have continuous maps $\Pi_j:\ X_{\theta,p}\to Y_j$,
hence the inclusion $[X_0,\tX_1]_{\theta,p}\subset X_{\theta,p}\cap \text{Ker}(\Pi_j)$. 
For the reverse inclusion
just write 
\begin{eqnarray*}
\|u\|_{[X_0,\tX_1]_{\theta,p}}^p\sim \int_0^1\left(\frac{N_J(t,\Lambda(t)u)}{t^{\theta}}\right)^p
\frac{dt}{t}
&\sim &\sum_{k=1}^n\int_0^1\frac{\|\Lambda(t)u\|_{Y_k}^p}{t^{(\theta-\theta_k)p+1}}
\frac{dt}{t}
\\
&\sim& \sum_{k=1}^n\sum_0^\infty \|\Lambda(2^{-j})u\|_{Y_k}^p2^{pj(\theta-\theta_k)}.
\end{eqnarray*}
Then each term $\|\Lambda(2^{-j})u\|_{Y_k}$ is decomposed as a telescopic series 
$$\sum_{m=0}^{j-1}\|(\Lambda(2^{-m})-\Lambda(2^{-(m+1)})u\|_{Y_k}\text{ if }\theta<\theta_k,
\text{ or }
\sum_{m\geq j}\|(\Lambda(2^{-m})-\Lambda(2^{-(m+1)})u\|_{Y_k}\text{ if }\theta>\theta_k.
$$
The rest of the proof is unchanged.
\end{proof}
\paragraph{Example: several boundary conditions in Sobolev spaces} This is a rather well-known
example : $X_0=L^2(\R^{+*}\times \R^{d-1}),\ X_1=H^n(\R^{+*}\times \R^{d-1})$, and the
subspaces are defined as $X_1^k=\{u\in H^n:\ \partial_1^ku|_{x_1=0}=0\}$, for a set of 
indices $k$ included in $\{0,\cdots,n-1\}$. All the $Y_k$ are $L^2(\R^{d-1})$, we use the maps
$u\to (\partial_1^ku)|_{x_1=0}$. 
\\It is easily seen (as in example $1$) 
that for $0\leq k\leq n-1$,  $X_1^k$ has interpolation order $(k+1/2)/n$. Moreover, the spaces 
are independent : 
given a function $g$ defined on the boundary $\R^{d-1}$, 
we define maps $(R_k)_{0\leq k\leq n-1}$ by the formulas 
$$
\cF'(R_k g)(x_1,\xi')=\frac{\widehat{g}(\xi')\varphi_k(\lambda(t,\xi')x_1)}{\lambda(t,\xi')^j},
$$
where $\di \lambda(t,\xi')=\frac{1+t^{1/n}\cxip}{t^{1/n}}$,
the functions $\varphi_k$ are in $C_c^\infty(\R^+)$ and such that
$$
\partial_1^j\varphi_k(0)=\delta_k^j.
$$
Each map $R_k$ produces a function $R_kg$ such that 
\begin{equation}\label{eq:indep}
(\partial_1^kR_kg)|_{x_1=0}=g,\ 
(\partial_1^kR_jg)|_{x_1=0}=0,\ k\neq j. 
\end{equation}
Hence for  $u\in H^n$, the linear map 
$u\to Lu:=\sum_{k=0}^{n-1}R_k(\partial_1^ku|_{x_1=0})$ satisfies the property
$$
\forall\,0\leq k\leq n-1,\ \partial_1^kLu|_{x_1=0}=\partial_1^ku|_{x_1=0},
$$
(in other words, $R$ is a right inverse to the trace operator).\\
A computation gives 
\begin{equation}\label{eq:order}
\|R_kg_k\|_2+t^2\|R_kg_k\|_{H^n}^2\sim t^{(2k+1)/n}\int_{\R^{d-1}}(1+t^2\cxip^{2n})^{1-(k+1/2)/n}|\widehat{g}(\xi')|^2\xi'.
\end{equation}
Take $\Lambda(t)u=\cF^{-1}(1_{|\xi'|\leq t^{-1/n}}\frac{\widehat{Eu}}{1+
t^{2/n}\xi_1^2})$, with $E$ a smooth extension $H^n(\R^{+*}\times \R^{d-1})\to H^n(\R^d)$. 
From the fact that 
$\partial_1^ju|_{x_1=0}$ is localized at frequencies $|\xi'|\leq t^{-1/n}$, 
we see that 
$$
\|R_j(\partial_1^ju|_{x_1=0})\|_{H^n}\lesssim t^{(j+1/2)/n}\|\partial_1^ju|_{x_1=0}\|_2,
$$
so properties \eqref{eq:indep} and \eqref{eq:order} give that the spaces $(X_1^k)_{0\leq k\leq n-1}$ are independent.
\\
We deduce the identification of $[X_0,\tX_1]_{\theta,p}$ for 
$\theta\notin\{(k+1/2)/n\}_{0\leq k\leq n-1}$:
$$
[X_0,\tX_1]_{\theta,p}=\left\{u\in [L^2,H^n]_{\theta,p}:\ \forall\,k\text{ such that }\ \frac{k+1/2}{n}<\theta,\ \partial_1^ku|_{x_1=0}=0\right\}.$$
\bibliographystyle{plain}
\bibliography{biblio}
\end{document}